\documentclass{amsart}
\usepackage{amsmath,amssymb,amsthm,amscd}
\newtheorem{theo}{Theorem}

\newtheorem{lemm}[theo]{Lemma}
\newtheorem{prop}[theo]{Proposition}
\newtheorem{cor}[theo]{Corollary}

\theoremstyle{definition}
\newtheorem{rema}[theo]{Remark}

\begin{document}

\title[Transitivity of endomorphisms]{Approximate unitary equivalence of finite index endomorphisms of the AFD factors}
\author{Koichi Shimada}
\email{shimada@ms.u-tokyo.ac.jp}
\address{Department of Mathematical Sciences
University of Tokyo, Komaba, Tokyo, 153-8914, Japan}
\date{}
\begin{abstract}
We consider  two finite index endomorphisms  $\rho$, $\sigma $ of any AFD factor $M$.
We characterize the condition for  there being a sequence $\{ u_n\}$ of unitaries of the factor $M$ with $\mathrm{Ad}u_n \circ \rho \to \sigma $.
The characterization is given by using the canonical extension of endomorphisms, which is introduced by Izumi.
Our result is a generalization of the characterization of approximate innerness of endomorphisms of the AFD factors, obtained by Kawahiashi--Sutherland--Takesaki and Masuda--Tomatsu.
Our proof, which does not depend on the types of factors, is based on recent development on the Rohlin property of flows on von Neumann algebras.
\end{abstract}

\maketitle

\section{Introduction}
In this paper, we characterize the approximate innerness of the difference of two finite index endomorphisms of the AFD factors of type III (Theorem \ref{main}).
More precisely, for  two finite index endomorphisms $\rho$ ,$ \sigma $ of any AFD factor $M$, we give a good necessary and sufficient condition for there being a sequence $\{ u_n\}$ of  unitaries of $M$ with $\mathrm{Ad}u_n \circ \rho \to \sigma $ as $n\to \infty $ in the sense of Masuda--Tomatsu \cite{MT1}.
First of all, we explain the reason why we are interested in this topic.
The reason is that this result should be useful for classifying group actions.
It has been important to classify group actions on von Neumann algebras  up to cocycle conjugacy.
Since a remarkable work of Connes \cite{C}, classification of group actions on the AFD factors has greatly been developed by many researchers.
In particular, the actions of discrete amenable groups on the AFD factors are completely classified  (See Jones \cite{J}, Ocneanu \cite{O}, Sutherland--Takesaki \cite{ST}, Kawahigashi--Takesaki--Sutherland \cite{KwST} and Katayama--Sutherland--Takesaki \cite{KtST}).
It is interesting to note that although there are many different actions up to conjugacy, they are clearly classified when we ignore the difference of  cocycle conjugacy.
One of the next problems is to classify actions of continuous groups.
Among them, classification of actions of compact groups is considered to be relatively easy because the dual of a compact group is discrete.
In fact, actions of compact abelian groups on the AFD factors have  completely  been classified by using this observation (See Jones--Takesaki \cite{JT} and Kawahigashi--Takesaki \cite{KwT}).
However, when it comes to classifying actions of non-abelian compact groups, the problem is much more difficult.
One of the reasons is that the dual action of an action of a non-abelian compact group is a collection of  endomorphisms, not of automorphisms.
Hence in order to proceed with classifying actions, it is important to understand the properties of endomorphisms.
In the proof of classification theorems of group actions, whether the difference of two actions is approximated by inner automorphisms or not is very important.
Hence  we should characterize the approximate innerness of the difference of two endomorphisms of the AFD factors.

In this paper, we characterize  the approximate innerness of the difference of two endomorphisms in the sense of Masuda--Tomatsu \cite{MT1} (Theorem \ref{main}).
In Masuda--Tomatsu \cite{MT3}, they propose a conjecture of the complete invariant for actions of discrete Kac algebras on the AFD factors (Conjecture 8.2  of \cite{MT2}).
The dual of minimal actions of compact groups are ones of them.
Our main theorem implies that if two actions of discrete Kac algebras on the AFD factors of type III have the same invariants, the difference of these two actions is approximately inner (See Problem 8.3 and the preceding argument to that problem of \cite{MT3}).
Our main theorem characterizes when one endomorphism transits to another endomorphism.
Hence the theorem may also be seen as a kind of endomorphism counterpart of the main theorem of Haagerup--St\o rmer \cite{HS}, which characterizes when one normal state of a von Neumann algebra transits to another normal state.
It is important to note that our main theorem is a generalization of Theorem 1 (1) of Kawahigashi--Sutherland--Takesaki \cite{KwST} and Theorem 3.15 of Masuda--Tomatsu \cite{MT1}.
The proof of our theorem is based on recent development on the Rohlin property of flows on von Neumann algebras, which does not depend on the types of the AFD factors.
Our method is also applicable to the characterization of the central triviality of automorphisms (Theorem 1. (2) of Kawahigashi--Sutherland--Takesaki \cite{KwST}).
In appendix, we give another proof of the characterization of the central triviality, which does not depend on the types of the AFD factors.

\textbf{Acknowledgment}
The author thanks Professor Reiji Tomatsu for introducing him to this topic and for giving him useful comments and Professor Toshihiko Masuda for pointing out a mistake in the first version of this paper.
The author is also thankful to Professor Yasuyuki Kawahigashi, who is his adviser, for his useful comments on the presentation of this work.
The author is supported by Research Fellowships of the Japanese Society for the Promotion of Science for Young Scientists No.26-6590. 
This work is also supported by the Program for Leading Graduate Schools, MEXT, Japan.

\section{Preliminaries}
\subsection{Notations}
Let $M$ be a von Neumann algebra. For a normal positive linear functional $\psi $ of $M$ and $x\in M$, set
\[ \| x\| _{\psi}:=\sqrt{\psi (x^*x)},\]
\[ \| x\| _{\psi}^\sharp :=\sqrt{\frac{\psi (x^*x)+\psi (xx^*)}{2}}.\]
\begin{lemm}
\label{inequality}
Let $\lambda $ be a $\sigma$-weakly continuous linear functional of a von Neumann algebra $M$ and $\lambda =\psi v$ be its polar decomposition.
Then we have 

\[ \| \lambda a\| \leq \psi (vaa^*v^*)^{1/2}\| \lambda \|^{1/2} , \]
\[ \| a\lambda \| \leq \psi (a^*a)^{1/2}\| \lambda \|^{1/2} \]
for any $a\in M$.
\end{lemm}
\begin{proof}
By Cauchy--Schwarz's inequality, for $x\in M$, we have
\begin{align*}
|\lambda a(x)|&=|\psi (vax)| \\
               &\leq \psi (vaa^*v^*)^{1/2}\psi (x^*x)^{1/2} \\
               &\leq \psi (vaa^*v^*)^{1/2} \| \lambda \|  ^{1/2} \| x\| . 
\end{align*}
The latter inequality is shown in a similar way.
\end{proof}

\subsection{A topology of semigroups of endomorphisms}
Let $M$ be a factor of type $\mathrm{III}$. 
Let $\mathrm{End}(M)_0$ be the set of all finite index endomorphisms $\rho$ of $M$. 
Let $d(\rho ) $ be the square root of the minimal index of $M\supset \rho (M)$ and $E_\rho $ be the minimal expectation from $M$ to $\rho (M)$. Set $\phi _\rho :=\rho ^{-1} \circ E_\rho$. 
In Masuda--Tomatsu \cite{MT1}, a topology of $\mathrm{End} (M)_0$ is introduced in the following way.
We have
\[ \rho _i \to \rho \]
if, by definition, $\| \psi \circ \phi _{\rho _i} -\psi \circ \phi _\rho \| \to 0$ for any $\psi \in M_*$. 
\subsection{Canonical extension of endomorphisms}
Let $\varphi $ be a normal faithful semifinite weight of $M$ and $\sigma ^\varphi$ be the group of modular automorphisms of $\varphi$.
In Izumi \cite{I}, an extension $\tilde{\rho}$ of $\rho \in \mathrm{End}(M)_0$ 
on the continuous core $\tilde{M}:=M\rtimes _{\sigma ^\varphi} \mathbf{R}$ is introduced in the following way.
We have
\[ \tilde {\rho} (x\lambda _t^{\sigma ^\varphi})=d(\rho )^{it}\rho (x) [D\varphi \circ \phi _{\rho} :D\varphi ]_t \lambda _t^{\sigma^\phi}\]
for $t\in \mathbf{R}$, $x\in M$, where $[D\varphi \circ \phi _{\rho}:D\varphi ]_t$ is the Connes cocycle between $\varphi \circ \phi _{\rho}$ and $\varphi$. 
This extension does not depend on the choice of $\varphi$ under a specific identification (See Theorem 2.4 of Izumi \cite{I}).
The extension $\tilde{\rho}$ is said to be the canonical extension of $\rho$.

In Lemma 3.5 of Masuda--Tomatsu \cite{MT1}, it is shown that there exists a left inverse $\phi _{\tilde{\rho}}$ of $\tilde{\rho}$ satisfying
\[ \phi _{\tilde{\rho}} (x\lambda ^\varphi _t )=d(\rho )^{-it} \phi _\rho (x[D\phi :D\phi \circ \phi _\rho ]_t )\lambda _t^\varphi \]
for $x\in M$, $t\in \mathbf{R}$.

\section{The main theorem}
The main theorem of this paper is the following.
\begin{theo}
\label{main}
Let $\rho$ , $\sigma $ be endomorphisms of an AFD factor $M$ of type $\mathrm{III}$ with $d(\rho ), d(\sigma )<\infty$. 
Then the following two conditions are equivalent.

\bigskip

\textup{(1)} We have  $\phi _{\tilde{\rho}}\circ \theta _{-\log (d(\rho) /d(\sigma ))}|_{\mathcal{Z}(\tilde{M})}=\phi _{\tilde{\sigma }}|_{\mathcal{Z}(\tilde{M})}$. 
 
\textup{(2)} There exists a sequence $\{ u_n\}$ of unitaries of $M$ with $\mathrm{Ad}u_n \circ \rho \to \sigma $ as $n\to \infty$.
\end{theo}
As a Corollary, we have the following result.
\begin{cor}
Let $M$ be an AFD factor and $R_0$ be the AFD factor of type $\mathrm{II}_1$.
Take endomorphisms $\rho _1$, $\rho _2\in \mathrm{End}(M)_0$ and $\sigma _1$, $\sigma _2\in ^mathrm{End}(R_0)_0$. 
Then the following two conditions are equivalent.

\bigskip

\textup{(1)} There exists a sequence of unitaries $\{u_n\}$ of $M \otimes R_0$ with $\mathrm{Ad}u_n \circ (\rho _1\otimes \sigma _1) \to \rho _2\otimes \sigma _2$ as $n\to \infty$.

\textup{(2)} There exists a sequence of unitaries $\{v_n\}$ of $M$ with $\mathrm{Ad}v_n \circ \rho _1 \to \rho _2$ as $n \to \infty$.
\end{cor}
\begin{proof}
Since $\sigma _1$ and $\sigma _2 $ are approximately inner, we may assume that $\sigma _1=\sigma _2=\mathrm{id}_{R_0}$.
By the identification $\mathcal{Z}((M\otimes R_0)\rtimes _{\sigma ^\varphi \otimes \mathrm{id}_{R_0}}\mathbf{R})\cong \mathcal{Z}((M\rtimes _{\sigma ^\varphi}\mathbf{R})\otimes R_0) \cong \mathcal{Z}(M\rtimes _{\sigma ^\varphi }\mathbf{R})$ by
\[ (x\otimes y)\lambda _t^{\sigma ^\varphi \otimes \mathrm{id}_{R_0}}\mapsto (x\lambda _t^{\sigma ^\varphi})\otimes y,\]
we have $\phi _{\rho _i \otimes \mathrm{id}_{R_0}}=\phi _{\rho_i}$ for $i=1,2$.
We also have $\d(\rho _i \otimes \mathrm{id}_{R_0})=d(\rho _i)$.
Hence by Theorem \ref{main}, conditions (1) and (2) are equivalent.
\end{proof}
Note that this corollary would be quite difficult to show without Theorem \ref{main} (See also Section 3 of Connes \cite{C2}).

Theorem \ref{main} should also be useful for classifying actions of compact groups on the AFD factors of type $\mathrm{III}$.
Popa--Wassermann \cite{PW} and Masuda--Tomatsu \cite{MT2} showed that any compact group has only one minimal action on the AFD factor of type $\mathrm{II}_1$, up to conjugacy.  
One of the next problems is to classify actions of compact groups on the AFD factors of type III.
In Masuda--Tomatsus \cite{MT3} and \cite{MT4}, they are trying to solve this problem, and some partial answers to this problem are obtained (Theorems A, B of \cite{MT3} and Theorem 2.4 of \cite{MT4}).
However, still the problem has not been solved completely.
In Masuda--Tomatsu \cite{MT3}, a conjecture about this classification problem is proposed (Conjecture 8.2).
Our main theorem implies that if two actions of discrete Kac algebras on the AFD factors of type III have the same invariants, the difference of these two actions is approximately inner (See Problem 8.3 and the preceding argument to that problem of Masuda--Tomatsu \cite{MT3}).
In order to classify group actions, whether the difference of two actions is approximately inner or not is very important.
Kawahigashi--Sutherland--Takesaki \cite{KwST} and Masuda--Tomatsu \cite{MT1} characterize the approximate innerness of endomorphisms under such a motivation.
Theorem \ref{main} is a generalization of their results.

In the following, we will show Theorem \ref{main}. 
Implication \textup{(2)}$\Rightarrow $ \textup{(1)} is shown easily by using known results.

\bigskip

\textit{Proof of implication \textup{(2)} $\Rightarrow$ \textup{(1)} of Theorem \ref{main}.}
This is shown by the same argument as that of the proof of implication (1) $\Rightarrow $ (2) of Theorem 3.15 of \cite{MT1}.
Assume that we have $\mathrm{Ad}u_n \circ \rho \to \sigma $ as $n \to \infty$.
Then by the continuity of normalized canonical extension (Theorem 3.8 of Masuda--Tomatsu \cite{MT1}), we have
\[ \phi _{\tilde{\rho}}\circ \theta _{-\log d(\rho)} \circ \mathrm{Ad}u_n^*(x) \to \phi _{\tilde{\sigma }} \circ \theta _{-\log d(\sigma )} (x)\]
in the strong* topology for any $x\in \tilde{M}$.
Hence we have
\[ \phi _{\tilde{\rho}}\circ \theta _{-\log (d(\rho )/d(\sigma ))}|_{\mathcal{Z}(\tilde{M})} =\phi _{\tilde{\sigma }}|_{\mathcal{Z}(\tilde{M})}.\]
\qed

\bigskip

In the following, we will show the opposite implication. 
Our strategy is to reduce the problem to that of endomorphisms on semifinite von Neumann algebras.
In order to achieve this, in Kawahigashi--Sutherland--Takesaki \cite{KwST} and Masuda--Tomatsu \cite{MT1}, they have used discrete decomposition theorems (See Connes \cite{C}).
However, in our situation, the centers of the images of caonical extensions may not coincide with that of $\tilde{M}$.
This makes the problem difficult.
It seems that Corollary 4.4 of Izumi \cite{I} means that it is difficult to show Theorem \ref{main} by the same strategy as those in them.
Instead, we will use continuous decomposition.
We also note that our method gives a proof of Theorem (1) of Kawahigashi--Sutherland--Takesaki \cite{KwST} which does not depend on the types of the AFD factors.

\section{Approximation on the continuous core}
In order to prove implication \textup{(1)} $\Rightarrow$ \textup{(2)} of  Theorem \ref{main}, we need to prepare some lemmas.
We first show the implication when $\phi _{\tilde{\rho}}=\phi _{\tilde{\sigma}}$. 
Until the end of the proof of Lemma \ref{22.5}, we always assume that $d(\rho)=d(\sigma)$ and $\phi _{\tilde{\rho}}=\phi _{\tilde{\sigma}}$.
Choose a dominant weight $\varphi $ of $M$ (For the definition of dominant weights, see Definition II.1.2. and Theorem II.1.3. of Connes--Takesaki \cite {CT}).
Then by Lemma 2.3 (3) of Izumi \cite{I}, it is possible to choose unitaries $u$ and $ v$ of $M$ so that  $ (\varphi , \mathrm{Ad} u\circ \rho )$ and $(\varphi , \mathrm{Ad}v \circ \sigma )$ are invariant pairs (See Definition 2.2 of Izumi \cite{I}). 
More precisely, we have
\[ \varphi \circ \mathrm{Ad}u \circ \rho =d(\rho ) \varphi , \ \varphi \circ E_{\mathrm{Ad}u\circ \rho} =\varphi ,\]
\[ \varphi \circ \mathrm{Ad} v \circ \sigma =d(\sigma ) \varphi , \ \varphi \circ E_{\mathrm{Ad}v\circ \sigma} =\varphi .\]
By replacing $\rho$ by $\mathrm{Ad} u \circ  \rho $ and $\sigma $ by $\mathrm{Ad}v\circ \sigma $ respectively, we may assume that $(\varphi , \rho)$ and $(\varphi , \sigma )$ are invariant pairs. 
In the rest of this paper, we identify $\tilde{M}$ with $M\rtimes _{\sigma ^\varphi}\mathbf{R}$.
Let $h$ be a positive self-adjoint operator affiliated to $\tilde{M}$ satisfying $h^{-it}=\lambda _t^\varphi$. 
Let $\tau$ be a trace of $\tilde{M}$ defined by $\hat{\varphi }(h\cdot)$.
\begin{lemm}
\label{expectation}
For $\rho \in \mathrm{End}(M)_0$, we have $\phi _{\tilde{\rho}}=\tilde{\rho}^{-1} \circ E_{\tilde{\rho}}$, where $E_{\tilde{\rho}}$ is the conditional expectation with respect to $\tau$.
\end{lemm}
\begin{proof}
For $x\in M$ and $t\in \mathbf{R}$, we have
\begin{align*}
\tilde{\rho} \circ \phi _{\tilde{\rho}} (x\lambda _t^\varphi ) &= \tilde{\rho } (d(\rho )^{-it} \phi _{\rho} (x[D\varphi : D\varphi \circ \phi _\rho ]_t ) \lambda _t^\varphi ) \\
                                                               &=d(\rho )^{it} d(\rho )^{-it} \rho (\phi _\rho (x[D\varphi :D\varphi \circ \phi _\rho ]_t ))[D\varphi \circ \phi _\rho : D\varphi ]_t \lambda _t^\varphi \\
                                                               &=E_\rho (x[D\varphi : D\varphi \circ \phi _\rho ]_t ) [D\varphi \circ \phi _\rho :D\varphi ]_t \lambda _t^\varphi
\end{align*}
Since $(\varphi , \rho )$ is an invariant pair, we have 
\[ [D\varphi \circ \phi _\rho :D\phi ]_t =d(\rho )^{-it}.\]
Hence we have 
\[E_{\rho } (x[D\varphi :D\varphi \circ \phi _\rho ]_t )[D\varphi \circ \circ \phi _\rho :D\varphi ]_t \lambda _t^\varphi = E_{\rho }(x) d(\rho ) ^{it} d(\rho ) ^{-it} \lambda _t^\varphi =E_\rho (x) \lambda _t^\varphi .\]
Hence by an argument of p.226 of Longo \cite{L}, it is shown that $\tilde{\rho}\circ \phi _{\tilde{\rho}}$ is the expectation with respect to $\tau$.
\end{proof}
\begin{lemm}
\label{trace}
For $\rho \in \mathrm{End}(M)_0$, we have $\tau \circ \phi _{\tilde{\rho}} =d(\rho ) ^{-1} \tau$.
\end{lemm}
\begin{proof}
By Lemma \ref{expectation}, we have $\phi _{\tilde{\rho}}=\tilde{\rho}^{-1} \circ E_{\tilde{\rho}}$. 
On the other hand, by Proposition 2.5 (4) of Izumi \cite{I}, we have $\tau \circ \tilde{\rho } =d(\rho ) \tau$.
Hence we have
\begin{align*}
\tau \circ \phi _{\tilde{\rho}} &=d(\rho) ^{-1} \tau \circ \tilde{\rho }\circ \phi _{\tilde{\rho}} \\
                                &=d(\rho ) ^{-1} \tau \circ \tilde{\rho} \circ \tilde{\rho} ^{-1} \circ E_{\tilde{\rho }} \\
                                &= d(\rho ) ^{-1} \tau \circ E_{\tilde{\rho}} \\
                                &= d(\rho ) ^{-1} \tau .
\end{align*}
\end{proof}
In the following, we identify $\mathcal{Z}(\tilde{M})$ with $L^\infty (X, \mu )$.
Let
\[ \tau =\int _X^\oplus \tau _x \ d\mu (x)\]
be the direct integral decomposition of $\tau$.
\begin{lemm}
\label{kkey}
Let $\rho , \sigma $ be elements of $\mathrm{End}(M)_0$. Assume that $\phi _{\tilde{\rho}}|_{\mathcal{Z}(\tilde{M})}=\phi _{\tilde{\sigma }}|_{\mathcal{Z}(\tilde{M})}$ and $d(\rho )=d(\sigma )$. For $a\in \tilde{M}_+$ with $\tau (a) <\infty$, set
\[ b:=\tilde{\rho} (a)=\int _X^\oplus b_x \ d\mu (x),\]
\[ c:=\tilde{\sigma } (a) =\int _X^\oplus c_x \ d\mu (x).\]
Then we have
\[\tau _x (b_x)=\tau _x(c_x)\]
for almost every $x\in X$.
\end{lemm}
\begin{proof}
Take an arbitrary positive element $z$ of $\mathcal{Z}(\tilde{M})_+$. Then we have
\begin{align*}
\tau (bz) &= \int _X \tau _x(b_xz_x) \ d\mu (x) \\
          &=\int _X \tau _x(b_x)z_x \ d\mu (x).
\end{align*}
Similarly, we have
\[ \tau (cz) =\int _X\tau _x(c_x)z_x \ d\mu (x).\]
On the other hand, by Lemma \ref{trace}, we have 
\begin{align*}
\tau (bz)&= d(\rho ) \tau \circ \phi _{\tilde{\rho}}(bz) \\
         &=d(\rho ) \tau \circ \phi _{\tilde{\rho}}(\tilde{\rho}(a)z) \\
         &=d(\rho ) \tau \circ \tilde{\rho }^{-1} \circ E_{\tilde{\rho}} (\tilde{\rho }(a)z) \\
         &=d(\rho ) \tau \circ \tilde{\rho } ^{-1} (\tilde{\rho} (a) E_{\tilde{\rho}}(z)) \\
         &=d(\rho ) \tau (a \phi _{\tilde{\rho}}(z)).
\end{align*}
Since we assume $d(\rho )=d(\sigma ) $ and $\phi _{\tilde{\rho}}|_{\mathcal{Z}(\tilde{M})}=\phi _{\tilde{\sigma}}|_{\mathcal{Z}(\tilde{M})}$, the last number of the above equality is $d(\sigma ) \tau (a\phi _{\tilde{\sigma }}(z))$, which is shown to be $\tau (cz)$ in a similar way. 
Hence we have
\[ \int _X \tau _x(b_x)z_x \ d\mu (x)=\int _X \tau _x(c_x) z_x \ d\mu (x).\]
Since the maps $x\mapsto \tau _x(b_x)$ and $x\mapsto \tau _x(c_x)$ are integrable functions and $z\in L^\infty (X, \mu )=L^1(X, \mu )^*$ is arbitrary, we have $\tau _x(b_x)=\tau _x(c_x)$ for almost every $x\in X$.
\end{proof}
Note that we have never used the assumption that $M$ is approximately finite up to this point. 
However, in order to show the following lemma, we need to assume that $M$ is approximately finite.
Let 
\[ \tilde{M}=\int ^\oplus _X\tilde{M}_x \ d\mu (x)\]
be the direct integral decomposition.
\begin{lemm}
\label{convergence in strong}
Let $M$ be an AFD factor of type $\mathrm{III}$ and $\rho$, $\sigma $ be as in Lemma \ref{kkey}. 
Then for almost every $x\in X$, there exist a factor $B_x$ of type $I_\infty$, a unitary $u$ of $\tilde{M}_x$ and a sequence $\{ u_n\}$ of unitaries of $\tilde{M}_x$ with the following properties.

\bigskip

\textup{(1)} The relative commutant $B_x'\cap \tilde{M}_x$ is finite.

\textup{(2)} There exists a sequence of unitaries $\{ v_n\}$ of $B_x'\cap \tilde{M}_x$ with $u_n =(v_n \otimes 1)u$, where we identify $\tilde{M}_x$ with $(B_x'\cap \tilde{M}_x)\otimes B_x$.

\textup{(3)} For almost every $x\in X$ and for any $a\in \tilde{M}$, we have $\mathrm{Ad}u_n ((\tilde{\rho} (a))_x) \to (\tilde{\sigma }(a))_x $ in the strong * topology.

\textup{(4)} We have $B_x\subset u(\tilde{\rho}(\tilde{M}))_xu^*\cap (\tilde{\sigma }(\tilde{M}))_x$.
\end{lemm}
\begin{proof}
Let $B_0 \subset \tilde{\rho} (\tilde{M})$ be a factor of type $\mathrm{I}_\infty$ with $Q:=\tilde{\rho} (\tilde{M}) \cap B_0'$ finite.
Let $\{ f_{ij}^0 \}$ be a matrix unit generating $B_0$.
We may assume that $\tau (f_{ii}^0)<\infty$.
Then since $(\tau \circ E_{\tilde{\rho}})_x((f_{11}^0)_x)<\infty$ for almost every $x\in X$, $P:=\tilde{M} \cap B_0'$ is also finite.
Then by Lemma \ref{kkey}, there exists a partial isometry $v$ of $\tilde{M}$ with $v^*v=\tilde{\rho }(f_{11}^0)$, $vv^*=\tilde{\sigma}(f_{11}^0)$.
Set
\[ u:=\sum _{j=1}^\infty \tilde{\sigma}(f_{j1}^0)v\tilde{\rho}(f_{1j}^0).\]
Then $u$ is a unitary of $\tilde{M}$ with $u\tilde{\sigma}(f_{ij}^0)u^*=\tilde{\rho}(f_{ij}^0)$.
Set 
\[ B:=\tilde{\sigma }(B_0)(=u\tilde{\rho}(B_0)u^*), \]
\[ f_{ij}:=\tilde{\sigma }(f_{ij}^0)(=u\tilde{\rho}(f_{ij}^0)u^*).\]
By replacing $\tilde{\rho} $ by $\mathrm{Ad}u \circ \tilde {\rho}$, we may assume that $\tilde{\rho}(f_{ij})=\tilde{\sigma }(f_{ij})$.
In the following, we identify $\tilde{M}$ with $P \otimes B$ and $P $ with $R\otimes \mathcal{Z}(\tilde{M})$, where $R$ is the AFD factor of type $\mathrm{II}_1$.
By the approximate finiteness of $R$ and $\mathcal{Z}(\tilde{M})$, there exists a sequence $\{ \{ e_{ij}^n\otimes a_k^n \}_{i,j,k}\}_{n=1}^\infty$ of systems of partial isometries of $P$ with the following properties.

\bigskip

(1) For each $n$, the system $\{ e_{ij}^n\}_{i,j}$ is a matrix unit of $R$.

(2) For each $n$, the system $\{ a_k^n\}_k$ is a partition of unity in $\mathcal{Z}(\tilde{M})$.

(3) For each $n$, $\{ e_{ij}^{n+1}\}_{i,j} $ is a refinement of $\{ e_{ij}^n\}_{i,j}$.

(4) For each $n$, $\{ a_k^{n+1}\}_k$ is a refinement of $\{ a_k^n\}_k$.

(5) We have $\bigvee _{n=1}^\infty \{ e^n_{ij} \otimes a^n_k\}_{i,j,k}'' =P$.

\bigskip

Fix a natural number $n$. 
Then by Lemma \ref{kkey}, we have
\[ \tau _x ((\tilde{\rho}(e^n_{11} \otimes a_k^n\otimes f_{11}))_x)=\tau _x((\tilde{\sigma }(e^n_{11} \otimes a_k^n \otimes f_{11}))_x)\]
for almost every $x\in X$. 
Hence for almost every $x\in X$, there exists a partial isometry $v_k^n$ of $P_x=(\tilde{\rho}(f_{11})\tilde{M}\tilde{\rho}(f_{11}))_x$ with
\[ {v_k^n}^*v_k^n =\tilde{\rho} (e^n_{11} \otimes a_k^n \otimes f_{11})_x, \ v_k^n {v_k^n }^* =\tilde{\sigma } (e^n_{11} \otimes a_k^n \otimes f_{11})_x.\]
Set
\[ v_n :=\sum _{k,j} \tilde{\sigma }(e_{j1}\otimes a_k^n \otimes f_{11})_xv_k^n \tilde{\rho }(e_{1j} \otimes a_k^n \otimes f_{11})_x.\]
Then $v_n$ is a unitary of $ \tilde{\rho} (f_{11})_x \tilde{M}_x \tilde{\rho} (f_{11})_x$ with
\[ v_n \tilde{\rho}(e_{ij}^n\otimes a_k^n \otimes f_{11} )_xv_n^* =\tilde{\sigma }(e_{ij}^n \otimes a_k^n \otimes f_{11})_x.\]
Hence for almost every $x\in X$, there exists a sequence $\{ v_n\} $ of unitaries of $P_x$ with 
\[ \mathrm{Ad}(v_n \otimes 1) (\tilde{\rho}(a)_x )\to \tilde{\sigma }(a)_x\]
for any $a\in \tilde{M}$.
\end{proof}
\begin{lemm}
\label{convergence in strong globaly} Let $M$, $\rho$ and $\sigma $ be as in Lemma \ref{convergence in strong}. 
Then there exist a unital subfactor $B$ of $\tilde{M}$, a unitary $u$ of $\tilde{M}$ and a sequence $\{ u_n\}$ of unitaries of $\tilde{M}$ with the following properties.

\bigskip

\textup{(1)} The factor $B$ is of type $I_\infty$.

\textup{(2)} The relative commutant $B'\cap \tilde{M}$ is finite.

\textup{(3)} There exists a sequence of unitaries $\{ v_n\}$ of $B'\cap \tilde{M}$ with $u_n =(v_n \otimes 1)u$, where we identify $\tilde{M}$ with $(B'\cap \tilde{M})\otimes B$.

\textup{(4)} For any $a\in \tilde{M}$, we have $\mathrm{Ad}u_n \circ \tilde{\rho} (a) \to \tilde{\sigma }(a) $ in the strong * topology.

\textup{(5)} We have $B\subset u\tilde{\rho}(\tilde{M})u^*\cap \tilde{\sigma }(\tilde{M})$.
\end{lemm}
\begin{proof}
This is shown by ``directly integrating'' the above lemma.
\end{proof}

The conclusion of Lemma \ref{convergence in strong globaly} means that $\mathrm{Ad}u_n \circ \tilde{\rho}$ converges to $\tilde{\sigma }$ point *strongly. 
However, this convergence is slightly weaker than the topology we consider.
We need to fill this gap.
In order to achieve this, the following criterion is very useful.
\begin{lemm}
\textup{(Lemma 3.8 of Masuda--Tomatsu \cite{MT2}).}
\label{criterion}
Let $\rho$ and $\rho _n$, $n\in \mathbf{N}$ be be endomorphisms of a von Neumann algebra $N$ with left inverses $\Phi$ and $\Phi _n$, $n\in \mathbf{N}$, respectively. 
Fix a normal faithful state $\phi $ of $N$. Then the following two conditions are equivalent.

\textup{(1)} We have $\lim _{n \to \infty } \| \psi \circ \Phi _n -\psi \circ \Phi \| =0 $ for all $\psi \in N_*$.

\textup{(2)} We have $\lim _{n\to \infty } \| \phi \circ \Phi _n -\phi \circ \Phi \| =0$ and $\lim _{n\to \infty} \rho _n (a) =\rho (a) $ for all $a \in N$.
\end{lemm}

Hence what we need to do is to find a normal faithful state of $\tilde{M}$ satisfying condition (2) of Lemma \ref{criterion}.

\begin{lemm}
\label{key}
Let $M$, $\rho$, $\sigma $ be as in Lemma \ref{convergence in strong}. 
Then  there exists a sequence of unitaries $u_n$ of $\tilde{M}$ with $\mathrm{Ad}u_n \circ \tilde{\rho} \to \tilde{\sigma} $. 
\end{lemm}
\begin{proof}
Take a subfactor $B$ of $\tilde{M}$, a unitary $u$ of $\tilde{M}$ and a sequence $\{ v_n\}$ of unitaries of $\tilde{M}$ as in Lemma \ref{convergence in strong globaly}.
By condition (5)  in Lemma \ref {convergence in strong globaly}, we have $u^*Bu \subset \tilde{\rho}(\tilde{M})$.
Set
\[ F:=\tilde{\rho}^{-1}(u^*Bu).\]
Then we have 
\[ \tilde{\rho}^{-1}\circ \mathrm{Ad}u^*(B)=F,\] 
\[ \tilde{\rho}^{-1}\circ \mathrm{Ad}u^*(B'\cap \mathrm{Ad}u\circ \tilde{\rho}(\tilde{M}))=F'\cap \tilde{M}.\]
We also have 
\[ \mathrm{Ad}u\circ E_{\tilde{\rho}}\circ \mathrm{Ad}u^*|_B=\mathrm{id}_B,\]
\[\mathrm{Ad}u \circ E_{\tilde{\rho}}\circ \mathrm{Ad}u^*(B'\cap \tilde{M})=B'\cap \mathrm{Ad}u \circ \tilde{\rho}(\tilde{M}).\]
Let $\{ f_{ij}\}$ be  a matrix unit generating $B$.
Set
\[ \overline{\tau} (a):=\tau (a\tilde{\rho}^{-1}(u^*f_{11}u))\]
for $a\in F'\cap \tilde{M}$, which is a  faithful normal finite trace of $F'\cap \tilde{M}$.
Let $\varphi$ be a normal faithful state of $F$.
Let $\Psi _F: \tilde{M}\to (F'\cap \tilde{M})\otimes F$ is the natural identification map.
Then by the above observation, for $a\in B'\cap \tilde{M}$ and $i$, $j$, we have
\begin{align*}
\ &(\overline{\tau} \otimes \varphi )\circ \Psi _F\circ \phi _{\tilde{\rho}} \circ \mathrm{Ad}u^*(af_{ij}) \\
&=(\overline{\tau} \otimes \varphi )\circ \Psi _F \circ (\tilde{\rho}^{-1}\circ \mathrm{Ad}u^*)\circ (\mathrm{Ad}u \circ E_{\tilde{\rho}}\circ \mathrm{Ad}u^*)(af_{ij}) \\
                                                                                 &=(\overline{\tau}\otimes \varphi ) \circ \Psi _F\circ (\tilde{\rho}^{-1}\circ \mathrm{Ad}u^*)((\mathrm{Ad}u \circ E_{\tilde{\rho}}\circ \mathrm{Ad}u^*|_{B'\cap \tilde{M}})(a)f_{ij})  \\
                                                                                 &=(\overline{\tau} \circ \phi _{\tilde{\rho}} \circ \mathrm{Ad}u^*) (a) (\varphi \circ \phi _{\tilde{\rho}} \circ \mathrm{Ad}u^*) (f_{ij}).
\end{align*}
Since $B \subset \tilde{\sigma }(\tilde{M})\cap \mathrm{Ad}u \circ \tilde{\rho}(\tilde{M})$, we have
\[ E_{\tilde{\sigma}}(af_{ij})=E_{\tilde{\sigma}}(a)f_{ij},\]
\[ \mathrm{Ad}u \circ E_{\tilde{\rho}}\circ \mathrm{Ad}u^*(af_{ij})=\mathrm{Ad}u \circ E_{\tilde{\rho}} \circ \mathrm{Ad}u^*(a)f_{ij}\]
for $a\in B'\cap \tilde{M}$.
Notice that $\tilde{\sigma }^{-1}(f_{ij})=\tilde{\rho}^{-1}(u^*f_{ij}u)$ by condition (3) of Lemma \ref{convergence in strong globaly}.
Then for any $a\in B'\cap \tilde{M}$, we have
\begin{align*}
\ &( \overline{\tau}\otimes \varphi )\circ \Psi _F \circ  \phi _{\tilde{\rho}}\circ \mathrm{Ad}u^* (v_n^*\otimes 1)(af_{ij}) \\
&= (\overline{\tau}\otimes \varphi ) \circ \Psi _F \circ \phi _{\tilde{\rho}}((u^*(v_n^*av_n)u)(u^*f_{ij}u)) \\
&= \overline{\tau} \circ \phi _{\tilde{\rho}}(u^*(v_n^*av_n)u)\varphi (\tilde{\rho}^{-1}(u^*f_{ij}u)) \\
&=\tau (\phi _{\tilde{\rho}} (u^*(v_n^*av_n )u)\tilde{\rho}^{-1}(u^*f_{11}u))\varphi (\tilde{\rho}^{-1}(u^*f_{ij}u)) \\
&=\tau \circ \phi _{\tilde{\rho}}(u^*(v_n^*av_n)f_{11}u)\varphi (\tilde{\rho}^{-1}(u^*f_{ij}u)) \\
&=d(\rho ) \tau (u^*(v_n^*av_n)f_{11}u)\varphi (\tilde{\rho}^{-1}(u^*f_{ij}u)) \\
&=d(\sigma ) \tau (af_{11}) \varphi (\tilde{\sigma }^{-1}(f_{ij})) \\
&=\tau (\phi _{\tilde{\sigma }}(a)\tilde{\sigma }^{-1}(f_{11})) \varphi (\tilde{\sigma }^{-1}(f_{ij})) \\
&=\tau (\phi _{\tilde{\sigma }}(a)\tilde{\rho}^{-1}(u^*f_{11}u)) \varphi (\tilde{\sigma }^{-1}(f_{ij})) \\ 
&= (\overline{\tau}\otimes \varphi ) \circ \Psi _F \circ \phi _{\tilde{\sigma}}(af_{ij}).
\end{align*}
Hence we have $( \overline{\tau}\otimes \varphi )\circ \Psi _F \circ \phi _{\tilde{\rho}} \circ \mathrm{Ad}(u^*(v_n \otimes 1)^*)=(\overline{\tau}\otimes \varphi )\circ \Psi _F\circ \phi _{\tilde{\sigma }}$ for any $n$.
Hence by Lemma \ref{convergence in strong globaly} and Lemma \ref{criterion}, we have $\mathrm{Ad}((v_n\otimes 1)u)\circ \tilde{\rho} \to \tilde{\sigma}$.
\end{proof}

\section{Averaging by the trace-scaling action}
In this section, we always assume that $M$ is an AFD factor of type III.
Let $\varphi $ be a dominant weight of $M$ and $\rho, \sigma \in \mathrm{End}(M)_0$ be finite index endomorphisms with $(\varphi , \rho)$ and $(\varphi , \sigma )$ invariant pairs.
Set 
\[ \tilde{M}:=M\rtimes _{\sigma ^{\varphi}} \mathbf{R}.\]
Let $\psi _0$ be a normal faithful state of $\tilde{M}$ and $\{ \psi _i\}_{i=1}^\infty $ be a norm dense sequence of the unit ball of $\tilde{M}_*$.
Let $\theta $ be the dual action on $\tilde{M} $ of $\sigma ^\varphi$.
We will replace the sequence $\{ u_n\}$ chosen in the previous section so that it is almost invariant by $\theta $.
In order to achieve this, we use  a property of $\theta $ which is said to be the Rohlin property.
In order to explain this property,  we first need to explain related things.
Let $\omega $ be an ultrafilter of $\mathbf{N}$.
A sequence $\{ [-1,1] \ni t \mapsto x_{n,t} \in \tilde{M}\}_{n=1}^\infty $ of  maps from $[-1,1]$ to $\tilde{M}$ is said to be $\omega $-equicontinuous if for any $\epsilon >0$, there exist an element $U\subset \mathbf{N}$ of $\omega $ and $\delta >0$ with $\| x_{n,t} -x_{n,s}\| <\epsilon $ for any $s,t\in [-1,1]$  with $|s-t|<\delta $, $n\in U$.
Set 
\[ \mathcal{C}:=\{ (x_n)\in l^\infty (\tilde{M}) \mid \| x_n \psi -\psi x_n \| \to 0 \ \mathrm{for} \ \mathrm{any} \ \psi \in \tilde{M}.\} ,\]
\[ \mathcal{C}_{\theta , \omega } :=\{ (x_n) \in \mathcal{C}_\omega \mid  \mathrm{ the} \ \mathrm{maps} \ \{ t\mapsto \theta _t(x_n)\}_{n=1}^\infty \ \mathrm{are} \ \omega \ \mathrm{equicontinuous}.\} ,\]
\[ \mathcal{I}_\omega :=\{ (x_n)\in l^\infty (\tilde{M}) \mid x_n \to 0 \ \mathrm{in } \ \mathrm{the } \ \mathrm{*strong} \ \mathrm{topology}.\} .\]
Then $\mathcal{I}_\omega $ is a (norm) closed ideal of $\mathcal{C}_{\theta , \omega }$, and the quotient $\tilde{M}_{\theta , \omega }:=\mathcal{C}_{\theta , \omega }/\mathcal{I}_\omega $ is a von Neumann algebra.
As mentioned in Masuda--Tomtasu \cite{MT5}, the action $\theta $ has the Rohlin property, that is, for any $R>0$, there exists a unitary $v$ of $\tilde{M}_{\theta ,\omega} $ with
\[ \theta _t(v)=e^{-iRt}v\]
for any $t\in \mathbf{R}$ (See Section 4 of Masuda--Tomatsu \cite{MT5}).
Choose arbitrary numbers $r>0$ and $0<\epsilon <1$.
Then since $M$ is of type III, there exists a real number $R$ which is not of the discrete spectrum of $\theta |_{\mathcal{Z}(\tilde{M})}$ and which satisfies $r/R<\epsilon ^2$.
Then as shown in Theorem 5.2 of Masuda--Tomatsu \cite{MT5}, there exists a normal injective *-homomorphism $\Theta$ from $\tilde{M}\otimes L^\infty ([-R,R])$ to $\tilde{M}^\omega $ satisfying $x\otimes f\mapsto xf(v)$ for any $x\in \tilde{M}$, $f \in L^\infty([-R,R])$.
For each $t\in \mathbf{R}$, set
\[ \gamma _t:L^\infty ([-R,R]) \ni f \mapsto f(\cdot -t)\in L^\infty ([-R,R]),\]
where we identify $[-R,R]$ with $\mathbf{R}/2R\mathbf{Z}$ as measured spaces.
Then the *-homomorphisms $\Theta $ and $\gamma _t$ satisfy 
\[ \Theta \circ (\theta _t\otimes \gamma _t ) =\theta _t\circ \Theta \]
(See Theorem 5.2 of Masuda--Tomatsu \cite{MT5}).
\begin{lemm}
\label{10}
For $\psi \in \tilde{M}_*$ and $x\otimes f\in \tilde{M}\otimes L^\infty ([-R,R])$, we have
\[ \psi ^\omega \circ \Theta =\psi \otimes \tau _{L^\infty },\]
where $\tau _{L^\infty}$ is the trace coming from the normalized Haar measure of $L^\infty ([-R,R])$.
\end{lemm}
\begin{proof}
Let $\{ v_n\}$ be a representing sequence of $v$.
Then we have
\begin{align*}
\psi ^\omega \circ \Theta (x\otimes f) &= \psi ^\omega (xf(v)) \\
                                       &=\lim _{n\to \omega } \psi (xf(v_n)) \\
                                       &=\psi (x) \lim _{n \to \omega }f(v_n) \\
                                       &=\psi (x)\tau _{L^\infty }(f) \\
                                       &=(\psi \otimes \tau _{L^\infty } )(x\otimes f).
\end{align*}
\end{proof}
Since the maps
\[ [-R,R]\ni t\mapsto \psi _i \circ \phi _{\tilde{\rho}}\circ \theta _t \in (\tilde{M})_*,\]
\[ [-R,R]\ni t\mapsto \psi _i \circ \phi _{\tilde{\sigma }}\circ \theta _t \in (\tilde{M})_*\]
are norm continuous, the union of their images
\[ \{ \psi _i \circ \phi _{\tilde{\rho}} \circ \theta _t \mid t\in [-R,R]\} \cup \{ \psi _i \circ \phi _{\tilde{\sigma }} \circ \theta _t\mid t\in [-R,R]\}\]
is compact.
Hence there exists a finite set $-R=t_0 <\cdots <t_J =R$ of $[-R,R]$ such that
\[ \| \psi _i \circ \phi _{\tilde{\rho}} \circ \theta _{t_j}-\psi _i \circ \phi _{\tilde{\rho}}\circ \theta _t \| <\epsilon ,\]
\[ \| \psi _i \circ \phi _{\tilde{\sigma }} \circ \theta _{t_j} -\psi _i \circ \phi _{\tilde{\sigma }} \circ \theta _t\| <\epsilon \]
for any $i=1, \cdots , n$, $j=0, \cdots , J-1$ and $t\in [t_j, t_{j+1}]$.
We may assume that $t_j=0$ for some $j$.
Then by Lemma \ref{key}, there exists a unitary $u$ of $\tilde{M}$ with
\[ \| \psi _i \circ \phi _{\tilde{\rho}}\circ \theta _{t_j}\circ \mathrm{Ad}u -\psi _i \circ \phi _{\tilde{\sigma }} \circ \theta _{t_j}\| <\epsilon\]
for any $j=0, \cdots , J-1$, $i=1, \cdots ,n$ (Notice that we used the fact that we have $\phi _{\tilde{\rho}}\circ \theta _{t_j}=\theta _{t_j} \circ \phi _{\tilde{\rho}}$ and that we have $\phi _{\tilde{\sigma }}\circ \theta _{t_j}=\theta _{t_j}\circ \phi _{\tilde{\sigma }}$ for any $j=0, \cdots , J-1$).
Set
\[ U:[-R,R]\ni t\mapsto \theta _t(u) \in \tilde{M},\]
which is a unitary of $\tilde{M}\otimes L^\infty ([-R,R])$.
\begin{lemm}
\label{11}
We have
\[ \| (\psi _i \circ \phi _{\tilde{\rho}})^\omega \circ \mathrm{Ad} \Theta (U)|_{\mathrm{Im}\Theta } -(\psi _i \circ \phi _{\tilde{\sigma}})^\omega | _{\mathrm{Im}\Theta } \| <3\epsilon .\]
\end{lemm} 
\begin{proof}
Let $m$ be the normalized Haar measure of $[-R,R]$.
By Lemma \ref{10}, we have 
\begin{align*}
\ &\| (\psi _i \circ \phi _{\tilde{\rho}})^\omega \circ \mathrm{Ad}\Theta (U) |_{\mathrm{Im}\Theta } - (\psi _i \circ \phi _{\tilde{\sigma }})^\omega |_{\mathrm{Im}\Theta }\| \\
&=\| ((\psi _i \circ \phi _{\tilde{\rho}})\otimes \tau _{L^\infty}) \circ \mathrm{Ad}U -(\psi _i \circ \phi _{\tilde{\sigma }} ) \otimes \tau _{L^\infty } \| \\
&= \int _{[-R,R]} \| (\psi _i \circ \phi _{\tilde{\rho}} ) \circ \mathrm{Ad}\theta _t(u) -\psi _i \circ \phi _{\tilde{\sigma }} \| \ dm(t) \\
&=\int _{[-R,R]} \| (\psi _i \circ \phi _{\tilde{\rho}}) \circ \theta _t\circ \mathrm{Ad}u -\psi _i \circ \phi _{\tilde{\sigma }} \circ \theta _t \| \ dm(t) \\
&\leq \sum _{j=0}^{J-1} \int _{[t_j, t_{j+1}]} ( \| (\psi _i \circ \phi _{\tilde{\rho}} )\circ \theta _t -(\psi _i \circ \phi _{\tilde{\rho}}) \circ \theta _{t_j} \| \\
&+ \| \psi _i \circ \phi _{\tilde{\rho}}\circ \theta _{t_j} \circ \mathrm{Ad}u-\psi _i \circ \phi _{\tilde{\sigma }} \circ \theta _{t_j}\| \\
&+ \| \psi _i \circ \phi _{\tilde{\sigma }} \circ \theta _{t_j}-\psi _i \circ \phi _{\tilde{\sigma }} \circ \theta _{t} \| ) \ dm(t) \\
&\leq \sum _{j=0}^{J-1} \int _{[t_j, t_{j+1}]} (\epsilon + \epsilon +\epsilon ) \ dm(t) \\
&=3\epsilon .
\end{align*}
\end{proof}
\begin{lemm}
\label{12}
We have
\[ \| \theta _s(\Theta (U))-\Theta (U) \| _{\psi _0^\omega }^\sharp <2\epsilon \]
for $|s|\leq r$.
\end{lemm}
\begin{proof}
Notice that we have 
\[ (\theta _s\otimes \gamma _s) (U):t\mapsto \theta _s(U_{t-s}),\]
where $U_t$ denotes the evaluation of the function $U$ at the point $t$.
Hence by the definition of $U$, we have 
\[ (\theta _s\otimes \gamma _s)(U) _t=\theta _t(u)\]
for any $t\in [-R+r, R-r]$, where the left hand side is the evaluation of the function $(\theta _s\otimes \gamma _s) (U)$ at the point $t$.
Hence by Lemma \ref{10}, we have 
\begin{align*}
\ &\| \theta _s(\Theta (U))-\Theta (U) \| _{\psi _0^\omega }^\sharp \\
&= \| (\theta _s\otimes \gamma _s )(U) -U \| _{\psi _0 \otimes \tau _{L^\infty}}^\sharp \\
&=(\int _{[-R,R]} (\| ((\theta _s \otimes \gamma _s )(U))_t -U_t \| _{\psi _0}^\sharp )^2 \ dm(t) )^{1/2} \\
&\leq (\int _{[-R,-R+r]\cup [R-r,R]} 4 \ dm(t) ) ^{1/2} \\
&\leq (4\epsilon ^2 )^{1/2} \\
&=2\epsilon .
\end{align*}
\end{proof}
Let 
\[ \psi _i \circ \phi _{\tilde{\sigma }} =| \psi _i \circ \phi _{\tilde{\sigma }}|v_i \]
be the polar decompositions of $\psi _i \circ \phi _{\tilde{\sigma }}$ for $i=1, \cdots , n$.
\begin{lemm}
\label{13}
There exists a finite subset $-R=s_0< \cdots < s_K=R$ of $[-R,R]$ with the following properties.

\bigskip

\textup{(1)} We have
\[ \| (U-\sum _{k=0}^{K-1} \theta _{s_k}(u)e_k )v_i^* \| _{|\psi _i \circ \phi _{\tilde{\sigma }}|\otimes \tau _{L^\infty}} ^\sharp <\epsilon \]
for any $i=1, \cdots , n$, where $e_k:=\chi _{[s_k,s_{k+1}]}\in L^\infty ([-R,R])$.

\textup{(2)} We have 
\[ \| U-\sum _{k=0}^{K-1} \theta _{s_k}(u) e_k \| _{|\psi _i \circ \phi _{\tilde{\rho}}|\otimes \tau _{L^\infty}}^\sharp <\epsilon \]
for any $i=1, \cdots , n$.

\textup{(3)} We have
\[ \| U-\sum _{k=0}^{K-1} \theta _{s_k} (u)e_k\| _{(\psi _0 \circ \theta _{t_j}) \otimes \tau _{L^\infty}}^\sharp <\epsilon\]
for any $i=1, \cdots , n$ and $j=0, \cdots , J-1$.
\end{lemm}
\begin{proof}
Since the map $t\mapsto \theta _t(u)$ is continuous in the strong * topology, there exists a finite set $-R=s_0< \cdots < s_K=R$ of $[-R,R]$ with
\[ \| (\theta _t(u)-\theta _{s_k}(u) )v_i^* \| _{|\psi _i \circ \phi _{\tilde{\sigma }}| }^\sharp <\epsilon \]
for $i=1, \cdots , n$, $k=0, \cdots , K-1$ and $t\in [s_k, s_{k+1}]$,
\[ \| \theta _t(u)-\theta _{s_k}(u) \| _{|\psi _i \circ \phi _{\tilde{\rho}}|} ^\sharp <\epsilon \]
for $i=1, \cdots ,n$, $k=0, \cdots , K-1$ and $t\in [s_k, s_{k+1}]$, 
\[ \| \theta _t(u)-\theta _{s_k}(u)\| _{\psi _0\circ \theta _{t_j}}^\sharp <\epsilon \]
for $j=0, \cdots , J-1$, $k=0, \cdots ,K-1$ and $t\in [s_k, s_{k+1}]$.
Then we have 
\begin{align*}
\ &\| (U-\sum _{k=0}^{K-1}\theta _{s_k}(u)e_k )v_i^* \| _{|\psi _i \circ \phi _{\tilde{\sigma }}|\otimes \tau _{L^\infty} }^\sharp \\
&=(\sum _{k=0}^{K-1} \int _{[s_k, s_{k+1})} (\|  (\theta _t(u) -\theta _{s_k}(u))v_i^* \| _{|\psi _i \circ \phi _{\tilde{\sigma }}|}^\sharp )^2 \ dm(t) )^{1/2} \\
&< (\sum _{k=0}^{K-1} \int _{[s_k,s_{k+1})} \epsilon ^2 \ dm(t ))^{1/2} \\
&=\epsilon . 
\end{align*}
The other inequalities are shown in  a similar way.
\end{proof}
Set 
\[ V:=\sum _{k=0}^{K-1} \theta _{s_k}(u)e_k.\]
Take a representing sequence $\{ e_k^n\}_{n=1}^\infty$ of $\Theta (e_k)$ so that $\{ e_k^n\}_{k=0}^{K-1}$ is a partition of unity in $\tilde{M}$ by projections for each $n$. 
Set
\[ v_n:=\sum _{k=0}^{K-1} \theta _{s_k}(u)e_k^n,\]
which is a unitary.
The sequence $\{ v_n\}_{n=1}^\infty$ represents the unitary $\Theta (V)$.
Let $\{ u_n\}_{n=1}^\infty $ be a representing sequence of $\Theta (U)$.
\begin{lemm}
\label{P}
We have
\[ \lim _{n\to \omega } \| \theta _t (v_n)-v_n \| _{\psi _0}^\sharp <6\sqrt{\epsilon} .\]
for $t\in [-r,r]$.
\end{lemm}
\begin{proof}
Note that we have
\begin{align*}
\ &(\| \theta _t(a)\| _{\psi _0}^\sharp )^2 \\
&=\frac{1}{2}\psi _0 \circ \theta _t(a^*a+aa^*) \\
&=\frac{1}{2}(\psi _0 \circ \theta _{t_j}(a^*a+aa^*)) -\frac{1}{2}((\psi _0\circ \theta _{t_j}-\psi _0\circ \theta _t)(a^*a+aa^*)) \\
&\leq (\| a\| _{\psi _0\circ \theta _{t_j}}^\sharp )^2 +\| a\| ^2 \| \psi _0\circ \theta _{t_j}-\psi _0\circ \theta _t\|
\end{align*}
for any $a\in \tilde{M}$.
Hence for $t\in [t_j, t_{j+1}]\cap [-r,r]$, we have
\begin{align*}
\ &\| \theta _t(v_n) -v_n \| _{\psi _0}^\sharp \\
&\leq \| \theta _t(v_n-u_n) \| _{\psi _0}^\sharp +\| \theta _t(u_n)-u_n\| _{\psi _0}^\sharp +\| u_n-v_n \| _{\psi _0}^\sharp \\
&\leq (4\| \psi _0 \circ \theta _{t_j}-\psi _0 \circ \theta _t\| +(\| v_n -u_n\| _{\psi _0\circ \theta _{t_j}}^\sharp  )^2)^{1/2} \\
&+\| \theta _t(u_n)-u_n \| _{\psi _0}^\sharp +\| u_n -v_n \| _{\psi _0}^\sharp \\
&< (4\epsilon +(\| v_n -u_n\| _{\psi _0\circ \theta _{t_j}}^\sharp )^2 )^{1/2} \\
&+\| \theta _t(u_n)-u_n \| _{\psi _0}^\sharp + \| u_n -v_n \| _{\psi _0}^\sharp .
\end{align*}
Hence by Lemmas \ref{12} and \ref{13} (3), we have
\begin{align*}
\ &\lim _{n\to \omega} \| \theta _t(v_n) -v_n \| _{\psi _0}^\sharp \\
&\leq (4\epsilon +(\| V-U\| _{(\psi _0 \circ \theta _{t_j})\otimes \tau _{L^\infty}}^\sharp )^2)^{1/2} \\
&+\| \theta _t(U)-U\| _{\psi _0\otimes \tau _{L^\infty}}^\sharp +\| U-V\|_{\psi _0\otimes \tau _{L^\infty}}^\sharp \\
&<(4\epsilon +\epsilon ^2)^{1/2} +2\epsilon +\epsilon \\
&<6\sqrt{\epsilon} .
\end{align*}
\end{proof}
\begin{lemm}
\label{Q}
 We have
\[\lim _{n\to \omega } \| v_n ^*\psi _i \circ \phi _{\tilde{\rho}}-\psi _i \circ \phi _{\tilde{\sigma }}v_n^* \| \leq 7\epsilon \]
for any $i=1, \cdots , n$.
\end{lemm}
\begin{proof}
Notice that we have
\begin{align*}
\ & \| u_n^* (\psi _i \circ \phi _{\tilde{\rho}})-(\psi _i \circ\phi _{\tilde{\sigma }})u_n ^*\| \\
 & \| \Theta (U)^* (\psi _i \circ \phi_{\tilde{\rho}})^\omega |_M -(\psi _i \circ \phi _{\tilde{\sigma }})^\omega \Theta (U)^*|_M \| \\
 &\leq \| (\psi _i \circ \phi _{\tilde{\rho}} )^\omega \circ \mathrm{Ad}\Theta (U)|_{\mathrm{Im}\Theta}-(\psi _i \circ \phi _{\tilde{\sigma }})^\omega |_{\mathrm{Im}\Theta}\| .
 \end{align*}
Hence by Lemmas \ref{11} and \ref{13} (1) (2), we have
\begin{align*}
\ &\lim _{n\to \omega } \| v_n^*\psi _i \circ \phi _{\tilde{\rho}}-\psi _i \circ \phi _{\tilde{\sigma }}v_n^* \| \\
&\leq \lim _{n\to \omega }(\| (v_n^*-u_n^*)\psi _i \circ \phi _{\tilde{\rho}} \| \\
&+\| u_n^*\psi _i \circ \phi _{\tilde{\rho}}-\psi _i \circ \phi _{\tilde{\sigma}}u_n^*\| +\| \psi _i \circ \phi _{\tilde{\sigma }}(u_n^*-v_n^*)\| ) \\
&\leq \lim _{n\to \omega }(\| (v_n-u_n)^*\| _{|\psi _i \circ \phi _{\tilde{\rho}}|} \\
&+\| (\psi _i \circ \phi _{\tilde{\rho}})^\omega \circ \mathrm{Ad}\Theta (U) |_{\mathrm{Im}\Theta} -(\psi _i \circ \phi _{\tilde{\sigma }})^\omega |_{\mathrm{Im}\Theta}\| \\
&+\| (v_n-u_n)v_i^*\| _{|\psi _i \circ \phi _{\tilde{\sigma }}|} )\\
&=\| V-U\| _{|\psi _i \circ \phi _{\tilde{\rho}}|\otimes \tau _{L^\infty}} \\
&+\| (\psi _i \circ \phi _{\tilde{\rho}})^\omega \circ \mathrm{Ad}\Theta (U) |_{\mathrm{Im}\Theta} -(\psi _i \circ \phi _{\tilde{\sigma }})^\omega |_{\mathrm{Im}\Theta}\| +\| (V-U)v_i^*\|_{|\psi _i \circ \phi _{\tilde{\sigma}}|\otimes \tau _{L^\infty}} \\
&\leq \epsilon +3\epsilon +\epsilon \\
&=5\epsilon .
\end{align*}
Note that in order to show the second inequality, we used Lemma \ref{inequality}.
\end{proof}
By Lemmas \ref{P} and \ref{Q}, we have the following proposition.
\begin{prop}
\label{R}
There exists a sequence $\{ v_n\}_{n=1}^\infty $ of unitaries of $\tilde{M}$ with
\[ \lim _{n\to \infty }\| \theta _t(v_n)-v_n\| _{\psi _0}^\sharp =0,\]
\[ \lim _{n\to \infty}\| v_n^*\psi _i \circ \phi _{\tilde{\rho}}-\psi _i \circ \phi _{\tilde{\sigma }}v_n^* \| =0\]
for any $i=1, 2, \cdots $.
\end{prop}

\section{Approximation on $\tilde{M}\rtimes _\theta \mathbf{R}$.}
\label{lifting}
Set
\[ n_{\tau}:=\{ x\in \tilde{M}\mid \tau (x^*x)<\infty \}.\]
\begin{lemm}
\label{33}
Let $L^2(\tilde{M})$ be the standard Hilbert space of $\tilde{M}$ and $\Lambda :n_\tau \to L^2(\tilde{M})$ be the canonical injection. 
For each $x\in n_\tau$, set $V _{\tilde{\rho}}(\Lambda (x)):=\sqrt{d(\rho )}^{-1}\Lambda (\tilde{\rho}(x))$.
Then $V_{\tilde{\rho}}$ defines an isometry of $L^2(\tilde{M})$ satisfying
\[ V_{\tilde{\rho}}^* xV_{\tilde{\rho}}=\phi _{\tilde{\rho}} (x)\]
for any $x\in \tilde{M}$.
\end{lemm}
\begin{proof}
Take $x\in n_\tau$.
Then by Lemma 2.5 (4) of Izumi \cite{I}, we have
\begin{align*}
\| V_{\tilde{\rho}} \Lambda (x)\| ^2 &=d(\rho )^{-1} \tau (\tilde{\rho }(x^*x) ) \\
                                  &=\tau (x^*x) =\| \Lambda (x)\| ^2.
\end{align*}
Hence $V_{\tilde{\rho}}$ defines an isometry of $L^2(\tilde{M})$. 
Next, we show the latter statement.
We have $V_{\tilde{\rho}}^* (\Lambda (x))=\sqrt{d(\rho )}\Lambda (\phi _{\tilde{\rho}}(x))$ because
\begin{align*}
\langle V_{\tilde{\rho}}^* \Lambda (x), \Lambda (y)\rangle &=\langle \Lambda (x), \sqrt{d(\rho )}^{-1} \Lambda (\tilde{\rho}(y))\rangle \\
                                                           &=\sqrt{d(\rho )}^{-1} \tau (\tilde{\rho}(y)^*x) \\
                                                           &=\sqrt{d(\rho )}\tau (y^*\phi _{\tilde{\rho}}(x)) \\
                                                           &=\langle \sqrt{d(\rho )} \Lambda (\phi _{\tilde{\rho}}(x)), \Lambda (y)\rangle
\end{align*}
for any $x,y \in n_\tau$.
In order to show the third equality of the above, we used Lemma \ref{trace}.
Hence for any $x\in \tilde{M}$ and $y\in n_\tau$, we have
\begin{align*}
V_{\tilde{\rho}}^*xV_{\tilde{\rho}}\Lambda (y)&= \sqrt{d(\rho )}^{-1}V_{\tilde{\rho}}^*\Lambda (x\tilde{\rho}(y)) \\
                                              &= \Lambda (\phi _{\tilde{\rho}}(x\tilde{\rho}(y))) \\
                                              &=\phi _{\tilde{\rho}}(x) \Lambda (y).
\end{align*}
\end{proof}

Let $\rho $ be an endomorphism of a von Neumann algebra $M$.
Then since its canonical extension $\tilde{\rho}$ satisfies $\tau \circ \tilde{\rho}=d(\rho )\tau$, the endomorphism $\tilde{\rho}$ extends to $\tilde{M}\rtimes _\theta \mathbf{R}$ by $\lambda _t^\theta \mapsto \lambda _t^\theta $ for any $t\in \mathbf{R}$.
We denote this extension by $\tilde{\tilde{\rho}}$.

\begin{lemm}
\label{18}
Let $\alpha $ and $\sigma $ be finite index endomorphisms of a separable infinite factor $M$ and $\varphi $ be a dominant weight of $M$.
Assume that there exists a sequence  $\{ u_n\}$ of unitaries of $\tilde{M} \rtimes _\theta \mathbf{R}$ with $\mathrm{Ad}u_n \circ \tilde{\tilde{\rho}} \to \tilde{\tilde{\sigma}}$ as $n \to \infty$.
Then there exists a sequence  $\{ v_n \}$ of unitaries of $M$ with $\mathrm{Ad}v_n \circ \rho \to \sigma $. 
\end{lemm}
\begin{proof}
Since $(\varphi , \rho )$ and $(\varphi  , \sigma )$ are invariant pairs, it is possible to identify $\tilde{\tilde{\rho}}$ with $\rho \otimes \mathrm{id}_{B(L^2\mathbf{R})}$ and $\tilde{\tilde{\sigma }}$ with $\sigma \otimes \mathrm{id}_{B(L^2\mathbf{R})}$ through Takesaki duality, respectively (It is possible to choose the same identification between $M\otimes B(L^2\mathbf{R}) $ and $\tilde{M}\rtimes _\theta \mathbf{R}$ for $\tilde{\tilde{\rho}}$ and $\tilde{\tilde{\sigma }}$. See the argument preceding to Lemma 3.10  of Masuda--Tomatsu \cite{MT1}).
Then by (the proof of) Lemma 3.11  of Masuda--Tomatsu \cite{MT1}, there exist an isomorphism $\pi $ from $M\otimes B(L^2\mathbf{R})$ to $M$ and unitaries $u_\rho$, $u_\sigma $ of $M$ satisfying 
\[ \pi \circ (\rho \otimes \mathrm{id}) \circ \pi ^{-1} =\mathrm{Ad}u_\rho \circ \rho , \]
\[ \pi \circ (\sigma \otimes \mathrm{id} ) \circ \pi ^{-1} =\mathrm{Ad}u_\sigma \circ \sigma \]
(Although in the statement of Lemma 3.11 of Masuda--TOmatsu \cite{MT1}, the isomorphism $\pi $ depends on the choice of $\rho$, $\pi $ turns out to be independent of $\rho $ by its proof).
Then we have
\begin{align*}
\ & \mathrm{Ad}(u_\sigma ^*\pi (u_n) u_\rho ) \circ \rho \\
&= \mathrm{Ad}(u_\sigma ^*\pi (u_n ) )\circ \pi \circ (\rho\otimes \mathrm{id}_{B(L^2\mathbf{R})}) \circ \pi ^{-1} \\
&=\mathrm{Ad}u_\sigma ^* \circ \pi \circ (\mathrm{Ad}u_n \circ (\rho \otimes \mathrm{id}_{B(L^2\mathbf{R})}) ) \circ \pi ^{-1} \\
&\to \mathrm{Ad}u_\sigma ^* \circ \pi \circ (\sigma \otimes \mathrm{id}_{B(L^2\mathbf{R})} )\circ \pi ^{-1} \\
&=\mathrm{Ad}u_\sigma ^* \circ (\mathrm{Ad}u_\sigma \circ \sigma ) \\
&=\sigma .
\end{align*}
\end{proof}

\begin{lemm}
\label{19}
Let $\rho $ be an endomorphism with finite index and with $(\varphi , \rho )$ an invariant pair.
Let $E_{\tilde{\tilde{\rho}}}$ be the minimal expectation from $\tilde{\tilde{M}}$ to $\tilde{\tilde{\rho}}(\tilde{\tilde{M}})$.
Then we have the following.

\textup{(1)} For each $x\in \tilde{M}$, we have $E_{\tilde{\tilde{\rho}}}(x)=E_{\tilde{\rho}}(x)$.

\textup{(2)} For any $s\in \mathbf{R}$, we have $E_{\tilde{\tilde{\rho}}}(\lambda _t^\theta )=\lambda _t^\theta$.
\end{lemm}
\begin{proof}
This is shown in the proof of  Theorem 4.1 of Longo \cite{L}.
\end{proof}
\begin{lemm}
\label{36}
For $\xi \in L^2(\mathbf{R}, \tilde{M})$, set 
\[ V_{\tilde{\tilde{\rho}}}(\xi )(s):=V_{\tilde{\rho}}(\xi (s)).\]
Then $V_{\tilde{\tilde{\rho}}}$ is an isometry of $L^2(\mathrm{R}, \tilde{M})$ satisfying
\[ V_{\tilde{\tilde{\rho}}}^*xV_{\tilde{\tilde{\rho}}}=\phi _{\tilde{\tilde{\rho}}}(x)\]
for any $x\in M$, where $\phi _{\tilde{\tilde{\rho}}}=\tilde{\tilde{\rho}}^{-1}\circ E_{\tilde{\tilde{\rho}}}$.
\end{lemm}
\begin{proof}
The first statement is shown by the following computation.
\begin{align*}
 \| V_{\tilde{\tilde{\rho}}}(\xi )\| ^2&=\int _{\mathbf{R}}\| V_{\tilde{\rho}}(\xi (s))\| ^2 \ d\mu (s) \\
                                             &=\int _{\mathbf{R}}\| \xi (s)\| ^2\ d\mu (s) \\
                                             &=\| \xi \| ^2
\end{align*}
for $\xi \in L^2(\mathbf{R}, \tilde{M})$.
Next, we show the latter statement.
Choose $x\in M$ and $\xi \in L^2(\mathbf{R}, \tilde{M})$.
Then we have
\begin{align*}
V_{\tilde{\tilde{\rho}}}^* \circ \pi _\theta (x) \circ V_{\tilde{\tilde{\rho}}}(\xi ) &= V_{\tilde{\tilde{\rho}} }^* \pi _\theta (x)(s\mapsto V_{\tilde{\rho}}(\xi (s))) \\
                                                                                   &=V_{\tilde{\tilde{\rho}}} ^* (s\mapsto \theta _{-s}(x)\circ V_{\tilde{\rho}}(\xi (s))) \\
                                                                                   &=(s\mapsto V_{\tilde{\rho}}^* \circ \theta _{-s}(x) \circ V_{\tilde{\rho}}(\xi (s))) \\
                                                                                   &=(s\mapsto \phi _{\tilde{\rho}}(\theta _{-s}(x))(\xi (s))) \\
                                                                                   &=(s\mapsto \theta _{-s}( \phi _{\tilde{\rho}}(x))(\xi (s))) \\
                                                                                   &=\pi _\theta (\phi _{\tilde{\rho}}(x))(\xi ) \\
                                                                                   &=\phi _{\tilde{\tilde{\rho}}} (\pi _\theta (x))(\xi ).
\end{align*}
In order to show the fourth equality of the above, we used Lemma \ref{33}. The last equality of the above follows from Lemma \ref{19}.
For $t\in \mathbf{R}$ and $\xi \in L^2(\mathbf{R},\tilde{M})$, we have
\begin{align*}
V_{\tilde{\tilde{\rho}}}^*\lambda _t^\theta V_{\tilde{\tilde{\rho}}}\xi &= V_{\tilde{\tilde{\rho}}} ^*(s\mapsto V _{\tilde{\rho}}(\xi (s-t)) \\
                                                                              &= s\mapsto V_{\tilde{\rho}}^*V_{\tilde{\rho}}(\xi (s-t)) \\
                                                                              &= \lambda _t^\theta (\xi ).
\end{align*}

Thus we are done.
\end{proof}
\begin{lemm}
\label{20}
Let $N$ be a von Neumann algebra and $\{ V_n\}_{n=0}^\infty $ be a sequence of isometries on the standard Hilbert space $L^2(N)$ such that for each $n$, the map $\Phi _n:N\ni x\mapsto V_n^*xV_n$ is a left inverse of an endomorphism of $N$.
Then the following two conditions are equivalent.

\textup{(1)}  The sequence of operators $\{ V_n\}_{n=1}^\infty $ converges to  $V_0$ strongly. 

\textup{(2)}  We have $\|\psi \circ \Phi _n -\psi \circ \Phi _0\| \to 0$ for any $\psi \in N_*$.
\end{lemm}
\begin{proof}
This is shown by the same argument as that of the proof of Lemma 3.3 of Masuda--Tomatsu \cite{MT1}.
\end{proof}

\begin{lemm}
\label{22.5}
Let $\{ u_n\} $ be a sequence of unitaries of $\tilde{M}$ satisfying the following conditions. 

\bigskip

\textup{(1)} We have $\mathrm{Ad}u_n \circ \tilde{\rho} \to \tilde{\sigma }$ as $n \to \infty$.

\textup{(2)} For any compact subset $F$ of $\mathbf{R}$, we have $\theta _t(u_n )-u_n \to 0$ uniformly for $t\in F$.

\bigskip

Then we have $\mathrm{Ad} u_n \circ \rho \to \sigma$.
\end{lemm}
\begin{proof}
By Lemmas \ref{18}, \ref{36} and \ref{20}, it is enough to show that $V_{\tilde{\tilde{\rho}}}u_ n^*\to V_{\tilde{\tilde{\sigma }}}$.
Notice that we have
\begin{align*}
\ &V_{\tilde{\tilde{\rho}} }u_n^*(\xi \otimes f) \\
&=(s\mapsto V_{\tilde{\rho}}(\theta _{-s}(u_n^*)(\xi ) )f(s))
\end{align*}
for any $\xi \in L^2(M)$ and $f\in L^2\mathbf{R}$.
Hence we have
\begin{align*}
\ &\| V_{\tilde{\tilde{\rho}}}u_n^*(\xi \otimes f) -V_{\tilde{\tilde{ \sigma}} } (\xi \otimes f)\| ^2 \\
       &=\int _{\mathbf{R}}\| V_{\tilde{\rho}}(\theta _{-s}(u_n^*)(\xi ))-V_{\tilde{\sigma}}(\xi ) \| ^2|f(s)|^2\ ds \\
       &\leq \int _{\mathbf{R}} \| (V_{\tilde{\rho}}((\theta _{-s}(u_n^*) -u_n^*) (\xi ))\| ^2 |f(s)|^2 \ ds +\int _{\mathbf{R}} \|  V_{\tilde{\rho}} (u_n^*(\xi )) -V_{\tilde{\sigma }}(\xi ) \| ^2 |f(s)|^2 \ ds \\
       &\to 0
\end{align*} 
by the Lebesgue dominant convergence theorem.
Note that in order to show the last convergence, we use Lemmas  \ref{R}, \ref{33} and \ref{20}.
\end{proof}

\section{The proof of the main theorem}

\begin{lemm}
\label{40}
Let $M$ be an AFD factor and $\sigma $ be a finite index endomorphism of $M$ with $d(\sigma )=d$. 
Then there exists an endomorphism $\lambda$ with the following properties.

\bigskip

\textup{(1)} The endomorphism $\lambda $ is approximately inner.

\textup{(2)} We have $d(\lambda )=d$.

\textup{(3)} The endomorphism  $\lambda$ has Connes--Takesaki module and it is $\theta _{-\log d}|_{\mathcal{Z}(\tilde{M})}$.
\end{lemm}
\begin{proof}
By the proof of Theorem 3 of  Kosaki--Longo \cite{KL}, there exists an endomorphism $\lambda _0$ of the AFD factor of type $\mathrm{II}_1 $ with $d(\lambda _0)=d$.
Then $\mathrm{id}_M \otimes \lambda _0$ is an endomorphism of $M$ with $d(\mathrm{id}\otimes \lambda _0)=d$ and with $\mathrm{mod}(\mathrm{id} \otimes \lambda _0)$ trivial.
Hence by the existence of a right inverse of the Connes--Takesaki module of automorphisms (See Sutherland--Takesaki \cite{ST3}), there exists an automorphism $\alpha $ of $M$ with $\mathrm{mod}(\alpha \circ \lambda _0)=\theta _{-\log (d)}$.
By Theorem 3.15 of Masuda--Tomatsu (or by the same argument of our paper), it is shown that $\lambda :=\alpha \circ \lambda _0$ is approximately inner.
\end{proof}

Now, we return to the proof of the main theorem.

\bigskip

\textit{Proof of implication \textup{(1)} $\Rightarrow $ \textup{(2)} of Theorem \ref{main}.}
Let $\rho , \sigma $ be endomorphisms of $\mathrm{End}(M)_0$ with the first condition of Theorem \ref{main}.
Then by Lemma \ref{40}, there exist  endomorphisms $\lambda , \mu  \in \mathrm{End}(M)_0$ with the following properties.

\bigskip

(1) We have $d(\lambda )=d(\sigma)$, $d(\mu ) =d(\rho )$.

(2) We have  $\tilde{\lambda}|_{\mathcal{Z}(\tilde{M})}=\theta _{-\log (d(\sigma ))}|_{\mathcal{Z}(\tilde{M})}$ and $\tilde{\mu}|_{\mathcal{Z}(\tilde{M})}=\theta _{-\log (d(\rho ))}|_{\mathcal{Z}(\tilde{M})}$.

(3) The endomorphisms $\lambda $ and $\mu $ are approximately inner.

\bigskip

By the second condition, we have
\begin{align*}
\phi _{\tilde{\rho}}\circ \phi _{\tilde{\lambda}}|_{\mathcal{Z}(\tilde{M})} &=\phi _{\tilde{\rho}} \circ \theta _{\log d(\sigma)} |_{\mathcal{Z}(\tilde{M})} \\
                                                                    &=\phi _{\tilde{\sigma }}\circ \theta _{-\log (d(\sigma )/d(\rho ))} \circ \theta _{\log d(\sigma)}|_{\mathcal{Z}(\tilde{M})} \\
                                                                    &=\phi _{\tilde{\sigma}}\circ \theta _{\log (d(\rho))}|_{\mathcal{Z}(\tilde{M})} \\
                                                                    &=\phi _{\tilde{\sigma }}\circ \phi _{\tilde{\mu }}|_{\mathcal{Z}(\tilde{M})}.
\end{align*}

Hence by replacing $\rho $ by $\lambda \circ \rho$ and $\sigma $ by $\mu \circ \sigma $ respectively, we may assume that $d(\rho)=d(\lambda )$ and $\phi _{\tilde{\rho}} |_{\mathcal{Z}(M)}=\phi _{\tilde{\sigma}} |_{\mathcal{Z}(M)}$.
By Proposition \ref{R}, there exists a sequence $\{ u_n\}$ of unitaries of $\tilde{M}$ satisfying the assumptions of Lemma \ref{22.5}.
Hence by Lemma \ref{22.5}, we have $\mathrm{Ad}u_n \circ \rho \to \sigma $.
\qed

\section{Appendix (A proof of the characterization of central triviality of automorphisms of the AFD factors)}
In this section, we will see that it is possible to give a proof of a characterization theorem of central triviality of automorphisms of the AFD factors by a similar strategy to the proof of Theorem \ref{main}, which is independent of the types of the  AFD factors.

Let $M$ be an AFD factor of type III. Let $\alpha $ be an automorphism of $M$ and $\tilde{\alpha}$ be its canonical extension.
Set 
\[ p:=\mathrm{min}\{ q\in \mathbf{N}\mid \tilde{\alpha }^q \ \mathrm{is} \ \mathrm{centrally} \ \mathrm{trivial}\} ,\]
\[ G:=\mathbf{Z}/p\mathbf{Z}.\]
\begin{lemm}
The action $\{ \tilde{\alpha }_n \circ \theta _t\}_{(n,t)\in G\times \mathbf{R}}$ of $G\times \mathbf{R}$  on $\tilde{M}_{\omega , \theta }$ is faithful.
\end{lemm}
\begin{proof}
We will show this lemma by contradiction.
Let $\varphi $ be a normal faithful state of $\tilde{M}$ and $\{ \psi _j\}_{j=1}^\infty$ be a norm dense sequence of the unit ball of $\tilde{M}_*$.
Assume that there existed a pair $(n ,t)\in (G\times \mathbf{R})\setminus \{(0,0)\} $ satisfying $\tilde{\alpha }_n \circ \theta _{-t}(a)=a$ for any $a\in \tilde{M}_{\omega , \theta }$.
Then the automorphism $\tilde{\alpha }_n \circ \theta _{-t}$ would be centrally non-trivial because $\tilde{\alpha }_n \circ \theta _{-t}$ is trace-scaling if $t\not =0$.
Hence there would exist an $x$ of $\tilde{M}_\omega $, which can never be of $\tilde{M}_{\omega, \theta }$, with $\tilde{\alpha} _n (x)\not =\theta _t(x)$ and with $\| x\|\leq 1$.
Take a representing sequence $\{ x_k\}$ of $x$ with $\| x_k\| \leq 1 $ for any $k$.
Then we would have
\begin{align*}
\ & \lim _{k\to \omega }\| \tilde{\alpha} _n (x_k) -\theta _t(x_k)\| _{\varphi \circ \theta _s}^\sharp \\
 &=\mathrm{weak}\lim _{k \to \omega }\frac{1}{2} ( |\tilde{\alpha } _n (x_k) -\theta _t(x_k )|^2 +|(\tilde{\alpha }_n (x_k)-\theta _t(x_k))^*|^2 ) \\
& =2\delta >0
\end{align*}
for some $\delta >0$ (The constant $\delta $ does not depend on the choice of $s\in \mathbf{R}$).
Then for each natural number $L$, there would exist $k \in \mathbf{N}$ satisfying the following three conditions.

\bigskip

(1) We have  $\| x_k\| \leq 1$.

(2) We have 
\[ \| \theta _t(x_k)\psi _j -\psi _j \theta _t(x_k)\| (=\| x_k (\psi _j\circ \theta _{t})-(\psi _j \circ \theta _{t})x_k\| )<\frac{1}{L}\]
for $j=1, \cdots , L$, $| t| \leq L$ (Use the compactness of $\{ \psi _j \circ \theta _t\mid t\in L\}$. See also the argument just after Lemma \ref{10}).

(3) We have
\[ \| \tilde{\alpha }_n (x_k)-\theta _t(x_k) \| _\varphi ^\sharp >\delta .\]

\bigskip

Let $\Theta :L^\infty ([-L,L], dm(t))\otimes (\tilde{M}, \varphi) \to (\tilde{M}_{\omega , \theta }, \varphi ^\omega )$ be the inclusion mentioned in Section 5 (an inclusion coming from the Rohlin property of $\theta$), where $dm(t)$ is the normalized Haar measure of $[-L,L]$.
Set
\[ \tilde{y}:=([-L,L]\ni s\mapsto \theta _s(x_k)) \in L^\infty ([-L,L], dm(s))\otimes \tilde{M},\]
\[ y:=\Theta (\tilde{y}).\]
Since we would have $\tilde{\alpha }_n\circ \theta _{-t}$ is trivial on $\tilde{M}_{\omega, \theta }$, we would have
\[ (\tilde{\alpha }_n (\Theta (f)))_s=\tilde{\alpha }_n (\Theta (f)_{s-t})\]
for $f\in L^\infty ([-L,L])\otimes \tilde{M}$ and $s\in [-L+t, L-t]$, where $f_s$ is the evaluation of the function $f$ at $s\in [-L,L]$.
Hence we would have
\begin{align*}
\| \tilde{\alpha }_n (y) -y \| _{\varphi ^\omega }^\sharp &\geq (\int _{[-L+t,L-t]}(\| \tilde{\alpha }_n (\theta _{s-t}(x_k)) -\theta _{s}(x_k) \| _\varphi ^\sharp )^2 \ ds \\
                                                                                                                   &-\int _{[-L,-L+t]\cup [L-t,L]} 2^2 \ ds)^{1/2} \\
                                                                                                                   &\geq (\int _{[-L,L]}\delta ^2 \ ds -\frac{4t}{L})^{1/2} \\
                                                                                                                   &=(\delta ^2-\frac{4t}{L})^{1/2} .
\end{align*}
Since we have 
\[ (\theta _r(y))_s=\theta _s(y)\]
for any $0<r<1$, $s\in [-L+r,L-r]$, we have
\begin{align*}
 \| \theta _r(y)-y\| _{\varphi ^\omega}^\sharp &=(\int _{[-L,L]}(\| (\theta _r(y))_s-y_s \| _{\varphi  }^\sharp )^2\ ds )^{1/2} \\
                                  &\leq (\int _{[-L, -L+1]\cup [L-1,L]}2^2\ ds )^{1/2} \\
                                  &=\frac{2}{\sqrt{L}}
 \end{align*}
 for $| r|\leq 1$.
 We also have 
 \begin{align*}
 \| [y,\psi _j ]\| &=\| [y, \psi  _j]| _{\Theta (\mathbf{C}\otimes \tilde{M})} \| \\
                                  &\leq \| [y, \psi _i ]|_{\Theta ( L^\infty ([-L,L])\otimes \tilde{M})} \| \\
                                  &=\int _{[-L,L]} \| [ \tilde{y}_s , \psi _i ]\| \ ds \\
                                  &= \int _{[-L,L]} \| [\theta _s(x_k), \psi _j]\| \ ds \\
                                  &<\int _{[-L,L]} \frac{1}{L} \ ds \\
                                  &=\frac{1}{L} 
 \end{align*}
 for $j=1, \cdots , L$.
 Hence there would exist a sequence $\{ y_l\}$ of $\tilde{M}$ with the following properties.
 
 \bigskip
 
 (1) We have $\| y_l\| \leq 1$.
 
 (2) We have $\| [y_l, \psi _j]\| \to 0$ for any $j=1, 2, \cdots $.
 
 (3) For any $j=1,2, \cdots $, we have $\| \theta _r(y_l)-y_l\| _{\varphi }^\sharp \to 0$ uniformly for $s\in [-1,1]$.
 
 (4) We have $\| \tilde{\alpha }_n(y_l)-\theta _t(y_l)\| _{\varphi }^\sharp \geq \delta /2$ for any $l$.
 
 \bigskip
 
 This would contradict the assumption that $\tilde{\alpha }_n \circ \theta _{-t}$ were trivial on $\tilde{M}_{\omega , \theta }$.
\end{proof}

\begin{lemm}
For each $p\in \hat{(G\times \mathbf{R})}=\hat{G}\otimes \mathbf{R}$, there exists a unitary $u$ of $\tilde{M}_{\omega , \theta }$ with $\tilde{\alpha }_n \circ \theta _t(u)=\langle (n,t),p\rangle u$ for any $(n,t)\in G\times \mathbf{R}$.
\end{lemm}
\begin{proof}
The proofs of Theorems 4.10 and 7.7 of Masuda--Tomatsu \cite{MT5} works in our case.
\end{proof}
\begin{lemm}
There exist a non-zero projection $e$ of $(\tilde{M}_{\omega , \theta})^\theta $ with $\tilde{\alpha}(e)$ orthogonal to $e$.
\end{lemm}
\begin{proof}
By the previous lemma, for each natural number $l$, there exists a unitary $u$ of $\tilde{M}_{\omega , \theta }$ with $\tilde{\alpha }(u)=e^{2\pi i/p}u$ and with $\theta _t(u)=e^{-it/l}u$ for any $t$.
Then there exists a spectral projection $e$ of $u$ with $\tilde{\alpha }(e)\leq 1-e$, $\tau ^\omega (e)=1/p$ and with $\tau ^\omega (e-\theta _t(e))\leq 1/(2l)$ for $|t|\leq 1$.
By the usual diagonal argument, it is possible to choose a desired projection.
\end{proof}
\begin{theo}
\textup{(See Theorem 1 (2) of Kawahigashi--Sutherland--Takesaki)}
For an automorphism $\alpha $ of $M$, $\alpha $ is centrally trivial if and only if its canonical extension is inner.
\end{theo}
\begin{proof}
First, assume that  $\tilde{\alpha }$ is not centrally trivial.
Then by the previous lemma, neither is $\tilde{\tilde{\alpha }}$.
Hence neither is $\alpha $ centrally trivial (See, for example, Lemmas 5.11 and 5.12 of Sutherland--Takesaki \cite{ST}).
The above argument means that if $\alpha $ is centrally trivial, then $\tilde{\alpha}$ is centrally trivial.
Since $\tilde{M}$ is of type II, any centrally trivial automorphism of $\tilde{M}$ is inner.
The opposite direction is trivial by the central triviality of a modular endomorphism group.
\end{proof}

\begin{rema}
Finally, we remark that by our results and the result of Masuda \cite{M}, if we admit that the AFD factors are completely classified by their flows of weights, it is possible to classify the actions of discrete amenable groups on the AFD factors without separating cases by the types of the factors.
\end{rema}

\end{document}